\documentclass[11pt]{amsart}
\usepackage{amsmath}
\usepackage{amsfonts}

\usepackage{amscd}
\usepackage{amsthm}
\usepackage{amssymb} \usepackage{latexsym}
\usepackage{eufrak}
\usepackage{euscript}
\usepackage{epsfig}
\usepackage{graphics}
\usepackage{array}
\usepackage{enumerate}
\usepackage{dsfont}
\usepackage{color}
\usepackage{wasysym}
\usepackage{hyperref}
\usepackage{pdfsync}

\newcommand{\red}{\textcolor{red}}

\newcommand{\bel}[1]{\begin{equation}\label{#1}}

\newcommand{\be}{\begin{equation}}

\newcommand{\ba}{\begin{eqnarray}}
\newcommand{\ea}{\end{eqnarray}}

\newcommand{\qe}{\end{equation}}
\newcommand{\R}{{\mathbb R}}

\newcommand{\Z}{{\mathbb Z}}

\newcommand{\asinh}{{\mbox{arcsinh}}}

\newcommand{\Hmm}[1]{\leavevmode{\marginpar{\tiny%
$\hbox to 0mm{\hspace*{-0.5mm}$\leftarrow$\hss}%
\vcenter{\vrule depth 0.1mm height 0.1mm width \the\marginparwidth}%
\hbox to
0mm{\hss$\rightarrow$\hspace*{-0.5mm}}$\\\relax\raggedright #1}}}

\theoremstyle{theorem}
\newtheorem{thm}{Theorem}[section]

\theoremstyle{example}
\newtheorem{example}{Example}[section]
\theoremstyle{corollary}
\newtheorem{coro}{Corollary}[section]
\theoremstyle{lemma}
\newtheorem{lemma}{Lemma}[section]
\theoremstyle{definition}
\newtheorem{defi}{Definition}[section]
\theoremstyle{proof}

\theoremstyle{remark}
\newtheorem{rem}{Remark}[section]

\begin{document}

\title{Sharp Davies-Gaffney-Grigor'yan Lemma on Graphs}
\author{Frank Bauer}
\email{fbauer@math.harvard.edu}
\address{Harvard University Department of Mathematics\\
One Oxford Street, Cambridge MA 02138, Usa  and \\
Max Planck Institute for Mathematics in the Sciences\\
Inselstrasse 22, 04103 Leipzig, Germany}
\author{Bobo Hua}
\email{bobohua@fudan.edu.cn}
\address{School of Mathematical Sciences, LMNS, Fudan University, Shanghai 200433, China; Shanghai Center for Mathematical Sciences, Fudan University, Shanghai 200433, China}
\author{Shing-Tung Yau}
\email{yau@math.harvard.edu}
\address{Harvard University Department of Mathematics\\
One Oxford Street, Cambridge MA 02138, Usa}

\thanks{F.B. was partially supported by the Alexander von Humboldt foundation. S.T.Y. acknowledges support by the University of Pennsylvania/Air Force Office of Scientific Research grant ``Geometry and Topology of Complex Networks",
Award \#561009/FA9550-13-1-0097 and NSF DMS 1308244  Nonlinear Analysis on Sympletic, Complex Manifolds, General Relativity, and graph. B.H. was supported by NSFC, grant no. 11401106.}

\begin{abstract}In this note, we prove the sharp Davies-Gaffney-Grigor'yan lemma for minimal heat kernels on graphs.
\end{abstract}
\maketitle

\section{Introduction} The Davies-Gaffney-Grigor'yan Lemma (DGG Lemma for short) is a powerful tool in geometric analysis that leads to strong heat kernel estimates and has many important applications. On manifolds one writes it in the form

\begin{lemma}[Davies-Gaffney-Grigor'yan]\label{DGGManifold}
Let $M$ be a complete Riemannian manifold and $p_t(x,y)$ the minimal heat kernel on $M$. For any two measurable  subsets $B_1$ and $B_2$ of $M$ and $t>0,$ we have
\begin{equation}\label{e:DGG Riemannian}\int_{B_1}\int_{B_2} p_t(x,y)d\mathrm{vol}(x)d\mathrm{vol}(y) \leq \sqrt{\mathrm{vol}(B_1)\mathrm{vol}(B_2)}\exp\left(-\lambda t - \frac{d^2(B_1,B_2)}{4t}\right),\end{equation}
where $\lambda$ is the greatest lower bound, i.e. the bottom, of the $\ell^2$-spectrum of the Laplacian on $M$ and $d(B_1,B_2)=\inf_{x_1\in B_1,x_2\in B_2}d(x_1,x_2)$ the distance between $B_1$ and $B_2$.
\end{lemma}

While it would be desireable to obtain a DGG Lemma on graphs, it is known that it fails to be true in this setting. This already follows from the explicit calculation of the heat kernel on the lattice $\Z$ by Pang \cite{Pang93}. More generally, a  result of Coulhon and Sikora \cite{Coulhon08} implies in the graph case that the DGG Lemma is equivalent to the finite propagation speed property of the wave equation. However, Friedman and Tillich \cite[pp.249]{FriedmanTillich04} showed that for graphs the wave equation does not have the finite propagation speed property.

Still the situation is not totally hopeless since all obstructions appear when time $t$ is small compared to the distance $d$. In fact, it was recently shown in \cite{KellerLenzMuenchSchmidtTelcs15} that for small time heat kernels on graphs behave roughly like $t^d$, whereas for large time one expects they behave similarly to heat kernels on manifolds.

In a previous publication \cite[Theorem~1.1]{BauerHuaYau14} we were able to prove a version of the DGG Lemma on graphs and used it to prove for the first time heat kernel estimates for graphs with negative curvature lower bounds. However the previous DGG Lemma was not sharp in the following sense: On one hand, we only obtained the estimate as $\exp\left(-\frac12\lambda t\right)$ for the term involving the bottom of the spectrum $\lambda$ which is not optimal due to the factor 1/2. On the other hand, for technical reasons we had to re-scale and shift time which resulted in a weak version of Gaussian type estimate even for large times.

In this note, we adopt another proof-strategy, initiated by Coulhon and Sikora \cite{Coulhon08}, to resolve these problems and derive a sharp version of the DGG Lemma. Moreover, this approach allows us to extend the DGG Lemma to the far-reaching setting, i.e. for unbounded Laplacians on infinite graphs equipped with intrinsic metrics. For precise definitions and the terminology used, we refer to Section~\ref{sec:2}.

In particular we prove:

\begin{thm}[Functional formulation of DGG Lemma on graphs]\label{thm:Davies}
  Let $(V,\mu,m)$ be a weighted graph with an intrinsic metric $\rho$ with finite jump size $s>0.$ Let $A,B$ be two subsets in $V$ and $f,g\in \ell^2_m$ with $\mathrm{supp} f\subset A, \mathrm{supp} g\subset B,$ then
  $$|\langle e^{t\Delta}f,g\rangle|\leq e^{-\lambda t-\zeta_s(t,\rho(A,B))}\|f\|_{\ell^2_m}\|g\|_{\ell^2_m},$$ where $\lambda$ is the bottom of the $\ell^2$-spectrum of Laplacian and $$\zeta_s(t,r)=\frac{1}{s^2}\left(rs\cdot\asinh{\frac{rs}{t}}-\sqrt{t^2+r^2s^2}+t\right),\quad t>0,r\geq 0.$$
\end{thm}
For any subsets $A,B$ in $V,$ by setting $f= \mathbf{1}_A$ and $g= \mathbf{1}_B$ (as characteristic functions), we get
\begin{coro}[DGG Lemma on graphs] Under the same assumptions as above
 \begin{equation}\label{eq:Davies}
    \sum_{y\in B}\sum_{x\in A}m_xm_yp_t(x,y)\leq \sqrt{m(A)m(B)}e^{-\lambda t-\zeta_s(t,\rho(A,B))},
  \end{equation}
  where $p_t(x,y)$ is the minimal heat kernel of the graph.
\end{coro}
The function $\zeta_s(t,r),$ for $s=1$, appeared already in a number of publications in the graph setting, see for example \cite{Davies93, Pang93, Delmotte99}. For the sharpness of our DGG Lemma, we refer to Section~\ref{sec:sharp}. It is not difficult to see that for large time, i.e. $t\gg r,$ $\zeta_1(t,r)$ behaves like $\frac{r^2}{2t}.$ Hence for large time the DGG Lemma yields the Gaussian type estimate for the heat kernel in form of $\exp\left(-\frac{d^2(B_1,B_2)}{2t}\right).$ At first glance, it seems that we gain a factor two in the Gaussian exponent compared to the Riemannian case, which sounds absurd. In fact, it is not contradictory because a natural choice of distance functions, to satisfy the condition $s=1,$ is the combinatorial distance and in order to make it an intrinsic metric, one usually needs to re-normalize the physical Laplacian $\Delta$ to a so-called normalized Laplacian, such as $\frac12\Delta$ in the case of the lattice $\Z,$ which finally results in a scaling change.

The DGG Lemma has various applications for heat kernel estimates. Among many, on manifolds combining it with the Harnack inequality, obtained by the gradient estimate technique, one derives pointwise heat kernel upper bound estimates in term of the curvature bounds \cite{LiYau86, Li12}. For the counterpart on graphs, we refer to \cite{bauer2015} and \cite[Theorem~1.2]{BauerHuaYau14}. In particular, along the same lines one gets the sharp decay estimates involving the bottom of the Laplacian spectrum, i.e. $\exp(-\lambda t),$ using this new DGG Lemma. Moreover, as a direct application, it yields the Davies' heat kernel estimate, \cite[Theorem~10]{Davies93}, by setting $A=\{x\},B=\{y\},$ for $x,y\in V.$
\begin{coro}[Davies]
  For a weighted graph $(V,\mu,m)$ with the normalized Laplacian,
  $$p_t(x,y)\leq \frac{1}{\sqrt{m_xm_y}}\exp(-\lambda t-\zeta_1(t,d(x,y))),$$ where $d$ is the combinatorial distance.
\end{coro}

$\mathbf{Acknowledgements.}$
 We thank Alexander Grigor'yan, Thierry Coulhon and Adam Sikora for many fruitful discussions on Davies-Gaffney-Grigor'yan Lemma on manifolds and metric measure spaces.


\section{Setting and definitions}\label{sec:2}
\subsection{Weighted graphs}
We recall basic definitions for weighted graphs. Let $V$ be a countable discrete space serving as the set of vertices of a graph, $\mu:V\times V\ni(x,y)\mapsto \mu_{xy}\in [0,\infty)$ be an (edge) weight function satisfying
\begin{itemize}\item $\mu_{xy}=\mu_{yx},\ \ \ \ \quad\quad \forall x,y\in V,$ \item $\sum_{y\in V}\mu_{xy}<\infty,\quad \forall x\in V,$\end{itemize} and $m:V\ni x\mapsto m_x\in(0,\infty)$ be a measure on $V$ of full support. These induce a combinatorial (undirect) graph structure $(V,E)$ with the set of vertices $V$ and the set of edges $E$ such that for $x,y\in V,$ $\{x,y\}\in E$ if and only if $\mu_{xy}>0,$ in symbols $x\sim y.$
We refer to a triple $(V,\mu,m)$ a \emph{weighted graph} with the underlying graph $(V,E).$
Note that $(V,E)$ is not necessarily locally finite and possibly possesses self-loops. We denote by $d$ the combinatorial distance on it.

Given a weighted graph $(V,m,\mu),$ one can associate it with a Dirichlet form, see \cite{KellerLenz12}.
We denote by $C(V)$ the set of real functions on $V$ and by $C_c(V)$ the set of real functions with finite support. On the measure space $(V,m)$ we write $\ell^2(V,m),$ or simply $\ell^2_m,$ for the space of $\ell^2$-summable functions w.r.t. the measure $m.$ It becomes a Hilbert space if we equip it with an $\ell^2$-inner product $$\langle f,g\rangle:=\sum_{x\in V}f(x)g(x)m_x,\quad \forall f,g\in \ell^2_m.$$ The $\ell^2$ norm of a function $f\in\ell^2_m$ is given by $\|f\|_{\ell^2_m}:=\sqrt{\langle f,f\rangle}.$ Define the quadratic form $\widetilde{Q}:C(V)\to [0,\infty]$ given by
$$\widetilde{Q}(f):=\frac12\sum_{x,y\in V}\mu_{xy}|f(x)-f(y)|^2,\quad \forall f\in C(V).$$ The Dirichlet form $Q$ on $\ell^2_m$ is defined as the completion of $\widetilde{Q}|_{C_c(V)},$ the restriction of $\widetilde{Q}$ on $C_c(V)$, under the norm $$\|\cdot\|:=\|\cdot\|_{\ell^2_m}+\widetilde{Q}(\cdot).$$ We call the generator associated to $Q$ the Laplacian and denote it by $\Delta.$ In case that the underlying graph $(V,E)$ is locally finite, $C_c(V)$ lies in the domain of the generator and the Laplacian acts as, see \cite{KellerLenz10},
$$\Delta f (x) = \frac{1}{m_x} \sum_{y\in V} \mu_{xy}(f(y) - f(x)) \quad \forall f\in C_c(V).$$ The boundedness of the Laplacian, as a linear operator on $\ell^2_m,$ strongly depends on the choice of the measure $m.$ We call it the \emph{normalized Laplacian} if we set $m_x=\sum_{y\in V}\mu_{xy}$ for all $x\in V,$ and the \emph{physical Laplacian} if $m\equiv 1.$ The former is always a bounded operator on $\ell^2_m$ while the latter is possibly not.

In order to deal with unbounded Laplacians, it is often crucial to use so-called intrinsic metrics introduced in \cite{FrankLenzWingert12}.
\begin{defi}[Pseudo metric/intrinsic metric/jump size]\label{d:intrinsic} A \emph{pseudo metric} $\rho$ is a symmetric function, $\rho:V\times V\to[0,\infty),$ with zero diagonal which satisfies the triangle inequality. A pseudo metric $\rho$ on $V$ is called \emph{intrinsic} if
$\sum_{y\in V}\mu_{xy}\rho^{2}(x,y)\leq m_x, \forall x\in V.$ The \emph{jumps size} $s$ of a pseudo metric $\rho$ is given by $
    s:=\sup\{\rho(x,y)\mid x,y\in V, x\sim y\}\in[0,\infty].$
\end{defi} By the definition, one easily checks that the combinatorial distance $d$ is an intrinsic metric for the normalized Laplacian. For any weighted graph, one can always construct an intrinsic metric on it, see e.g. \cite{Huang11,HuangKellerMasamuneWojciechowski}.

\begin{defi}[Lipschitz function] We say that a function $f:V\to\R$ is \emph{Lipschitz} w.r.t. the (intrinsic) metric $\rho$ if $|f(x)-f(y)|\leq \kappa\rho(x,y)$ for all $ x,y\in V.$
The minimal constant $\kappa$ such that the above inequality holds is called the Lipschitz constant of $f$.
\end{defi}

\begin{defi}[Solution of the Dirichlet heat equation]
We say $u:[0,\infty)\times V \to \mathbb{R}$ solves the Dirichlet heat equation on the finite subset $\Omega \subset V$ if
\begin{equation*}\left\{ \begin{array}{r@{=}l@{\quad}l}
\frac{\partial}{\partial t}u(t,x) & \Delta_{\Omega}u(t,x), & x\in\Omega, t\geq 0,\\
u(0,x) & f(x),& x\in \Omega,
\\
u(t,x) & 0, & x \in V\setminus \Omega, t\geq 0. \end{array}\right.
\end{equation*}
where $\Delta_{\Omega}$ denotes the Laplacian with Dirichlet boundary condition, called Dirichlet Laplacian, on $\Omega,$ see e.g. \cite{Coulhon98,BauerHuaJost14}.
\end{defi}


\section{Integral maximum principle}\label{sec:Integral Max}
Throughout the rest of the paper we always assume that $(V,\mu,m)$ is a weighted graph with an intrinsic metric $\rho$ and finite jump size $s>0$. We begin with a simple lemma. For any $f\in C(V)$ and $x,y\in V,$ we denote by $\nabla_{xy}f:=f(y)-f(x)$ the difference of $f$ between $x$ and $y.$ By direct calculation, we have the following lemma.
\begin{lemma}\label{l:basic}
For two functions $f,g\in C(V)$ and $x,y\in V,$
\begin{enumerate}
\item[(a)]  $\nabla_{xy}(fg) = f(x) \nabla_{xy}g + g(y)\nabla_{xy}f = f(x) \nabla_{xy}g + g(x)\nabla_{xy}f  + \nabla_{xy}f\nabla_{xy}g$
\item[(b)]  $\nabla_{xy}e^{f}=(e^{\frac12f(x)}+e^{\frac12f(y)})\nabla_{xy}e^{\frac12f} $
\item[(c)] $|\nabla_{xy}(fe^{\frac12g})|^2 - \nabla_{xy}f\nabla_{xy}(fe^g) = f(x)f(y)|\nabla_{xy}e^{\frac12g}|^2.$
\end{enumerate}
\end{lemma}
We first state the integral maximum principle for Dirichlet Laplacians on finite subsets of a graph and extend it later to the whole graph. On Riemannian manifolds the integral maximum principle was introduced by Grigor'yan \cite{Grigoryan94}.
\begin{lemma}\label{l:monotonicity Delmotte finite}
Let  $\omega$ be a Lipschitz function on $V$ with Lipschitz constant $\kappa$ and assume that $f:[0,\infty)\times V\to\R$ solves the Dirichlet heat equation on a finite subset $\Omega\subset V$.  Then the function  $$\exp\left({2\lambda_1(\Omega)t-\frac{2}{s^2}(\cosh(\frac{\kappa s}{2})-1)t}\right)E_\Omega(t)$$ is nonincreasing in $t\in [0,\infty)$, where
$E_\Omega(t):= \sum_{x\in \Omega}m_xf^2(t,x)e^{\omega(x)}$ \red{,} and $\lambda_1(\Omega)$ is the first eigenvalue of Dirichlet Laplacian on $\Omega$.
\end{lemma}
\begin{proof}

 Since $f$ solves the Dirichlet heat equation on $\Omega$, $f(t,x)=0$ for any $x\in V\setminus \Omega$ and $t\geq 0,$ together with Green's formula (see e.g. \cite[Proposition~3.2]{KellerLenz10}, \cite[Lemma~4.7]{Haeseler11} or \cite[Lemma~2.4]{Schmidt12}) we obtain
 \begin{eqnarray}\label{eq:imp1}
    E_\Omega'(t)&=&\sum_{x\in V}m_x 2f(t,x)(\Delta_\Omega f(t,x)) e^{\omega(x)}=-\sum_{x,y\in V}\mu_{xy}\nabla_{xy}f\nabla_{xy}(fe^\omega)\nonumber\\&=&-\sum_{x,y\in V}\mu_{xy}|\nabla_{xy}(fe^{\frac12\omega})|^2 +\sum_{x,y\in V}\mu_{xy}(|\nabla_{xy}(fe^{\frac12\omega})|^2 - \nabla_{xy}f\nabla_{xy}(fe^{\omega})),
  \end{eqnarray}
 where we added zero in the last line. The first sum on the r.h.s. can be controlled from above by the Reighley quotient characterization of $\lambda_1(\Omega)$
  $$-\sum_{x,y\in V }\mu_{xy}|\nabla_{xy}(fe^{\frac12\omega})|^2\leq -2\lambda_1(\Omega)\sum_{x\in\Omega} m_x f(t,x)^2e^{\omega(x)}.$$ Applying Lemma \ref{l:basic} $(c)$, the second sum on the r.h.s. is given by
 \begin{eqnarray}\label{eq:eq1}
 &&\sum_{x,y\in\Omega}\mu_{xy}f(t,x)f(t,y)|\nabla_{xy}e^{\frac12\omega}|^2\nonumber
    \\&=&2\sum_{x,y\in\Omega}\mu_{xy}f(t,x)f(t,y)e^{\frac12(\omega(x)+\omega(y))}\left(\cosh\frac{\omega(y)-\omega(x)}{2}-1\right)\nonumber\\
    &\leq&\sum_{x,y\in\Omega}\mu_{xy}(f^2(t,x)e^{\omega(x)}+f^2(t,y)e^{\omega(y)})\left(\cosh\frac{\omega(y)-\omega(x)}{2}-1\right)\nonumber\\
    &=&2\sum_{x,y\in\Omega}\mu_{xy}f^2(t,x)e^{\omega(x)}\left(\cosh\frac{\omega(y)-\omega(x)}{2}-1\right),
  \end{eqnarray} where the last equality follows from the symmetry of the equation in $x$ and $y.$ Now we claim that for any neighbors $x\sim y,$ $$\cosh\frac{\omega(y)-\omega(x)}{2}-1\leq \rho(x,y)^2\frac{1}{s^2}(\cosh\frac{\kappa s}{2}-1).$$ It suffices to consider $x,y\in V$ such that $\rho(x,y)>0,$ otherwise it reduces to a trivial equation by the Lipschitz property of $\omega.$ The claim follows from
 $$ \cosh\frac{\omega(y)-\omega(x)}{2}-1 \leq\cosh\frac{\kappa \rho(x,y)}{2}-1 \leq\rho(x,y)^2\frac{1}{s^2}(\cosh\frac{\kappa s}{2}-1),$$ where we have used the monotonicity of the $\cosh$ function in the first, and the monotonicity of the function
  $$t\mapsto \frac{1}{t^2}(\cosh\frac{\kappa t}{2}-1),\quad t>0,$$ in the second inequality. Together with \eqref{eq:eq1} our claim implies that the second sum on the r.h.s. of \eqref{eq:imp1} can eventually be estimated from above by
  \begin{eqnarray*}&&2\sum_{x,y\in\Omega}\mu_{xy}f^2(t,x)e^{\omega(x)}\rho(x,y)^2\frac{1}{s^2}(\cosh\frac{\kappa s}{2}-1)\\
  &\leq &\frac{2}{s^2}(\cosh\frac{\kappa s}{2}-1)\sum_{x\in\Omega}m_xf^2(t,x)e^{\omega(x)},\end{eqnarray*} where we used that $\rho$ is an intrinsic metric. Combining everything, we get for any $t\geq 0,$
  $$E_\Omega'(t)\leq (-2\lambda_1(\Omega)+\frac{2}{s^2}(\cosh\frac{\kappa s}{2}-1))E_\Omega(t),$$ which implies the lemma.

\end{proof}

By a standard exhaustion argument, see e.g. \cite[Corollary~13.2]{Li12}, \cite[Section~3]{BauerHuaYau14} or \cite{KellerLenz10}, we obtain the integral maximum principle on the whole graph.

\begin{lemma}[Integral maximum principle]\label{l:monotonicity Delmotte} Let $\omega$ be a Lipschitz function on $V$ with Lipschitz constant $\kappa$ and $f(t,x):=e^{t\Delta}f_0(x)$ for some $f_0\in \ell^2_m.$ Then the function  $$\exp\left({2\lambda t-\frac{2}{s^2}(\cosh(\frac{\kappa s}{2})-1)t}\right)E(t)$$ is nonincreasing in $t\in [0,\infty)$,
where $E(t):= \sum_{x\in V}m_xf^2(t,x)e^{\omega(x)}$ and $\lambda$ is the bottom of the $\ell^2$-spectrum of Laplacian $\Delta.$
\end{lemma}

\begin{rem} Although it is possible that $E(t)=\infty$ for some $t\geq 0,$ the monotonicity still holds.
\end{rem}



\section{Proof of the main theorem}\label{sec:main theorem}
\begin{proof}[Proof of Theorem \ref{thm:Davies}]
Denote $r=\rho(A,B).$ For any $\kappa>0,$we set $\omega(x)=\kappa \rho(x,A),\forall x\in V.$ Then $\omega$ is a Lipschitz function with Lipschitz constant at most $\kappa$ and for any $h\in C(V)$
$$e^{\kappa r}\sum_{x\in B}m_xh^2(x)\leq \sum_{x\in B}m_xh^2(x)e^{\omega(x)}.$$ For $f\in\ell^2_m$ with $\mathrm{supp} f\subset A,$ let $f(t,x)=e^{t\Delta}f(x).$ Then the above inequality for $h(\cdot)=f(t,\cdot)$ and Lemma \ref{l:monotonicity Delmotte} yield
\begin{eqnarray*}\sum_{x\in B}m_xf^2(t,x)&\leq& e^{-\kappa r}E(t)\leq \exp\left(-2\lambda t+\frac{2}{s^2}(\cosh\frac{\kappa s}{2}-1)t-\kappa r\right)E(0)\\
&=&\exp\left(2(-\lambda t+\frac{1}{s^2}(\cosh\frac{\kappa s}{2}-1)t-\frac{\kappa}{2} r)\right)\sum_{x\in A}m_x f^2(x),\end{eqnarray*}
where we used that $\mathrm{supp} f\subset A.$  Since this estimate is true for all $\kappa>0$, we can choose, for fixed $s,t>0$ and $r\geq0,$ $\kappa$ such that the r.h.s. attains its minimum. One easily checks that $\zeta_s(t,r),$ defined in the introduction, is equal to
$$-\inf_{\kappa>0}\left(\frac{1}{s^2}(\cosh\frac{\kappa s}{2}-1)t-\frac{\kappa}{2} r\right).$$
Hence
$$\sum_{x\in B}m_xf^2(t,x)\leq e^{2(-\lambda t-\zeta_s(t,r))}\sum_{x\in A}m_x f^2(x).$$
That is, for all $f\in \ell^2_m$ with $\mathrm{supp} f\subset A$
$$\sup_{\substack{g\in \ell^2_m\\\mathrm{supp} g\subset B}}\frac{|\langle e^{t\Delta}f,g\rangle|^2}{\|g\|^2_{\ell^2_m}}=\sum_{x\in B}m_x|e^{t\Delta}f|^2\leq e^{2(-\lambda t-\zeta_s(t,r))}\|f\|_{\ell^2_m}^2.$$ This proves the theorem.\end{proof}

\section{Sharpness of the results}\label{sec:sharp}
The sharpness of the term $e^{-\zeta_s(t,\rho(A,B))}$ in \eqref{eq:Davies} in our DGG Lemma can be seen from Pang's result \cite[Theorem~3.5]{Pang93}.
\begin{example}[Pang]\label{example of Pang} Let $\mathbb{Z}$ be the unweighted standard one-dimensional lattice, i.e. an infinite line with each edge of weight one. For any $x,y\in \Z,$ set $d=d(x,y).$ The heat kernel for the normalized Laplacian satisfies, for some $C>1$,
\begin{eqnarray*}&&C^{-1}d^{-\frac{1}{2}}e^{-\zeta_1(t,d)}\leq p_t(x,y) \leq Cd^{-\frac{1}{2}}e^{-\zeta_1(t,d)},\quad \mathrm{for}\ 0<t\leq d\\&&C^{-1}t^{-\frac{1}{2}}e^{-\zeta_1(t,d)}\leq p_t(x,y) \leq Ct^{-\frac{1}{2}}e^{-\zeta_1(t,d)},\quad \mathrm{for}\ d\leq t.\end{eqnarray*}
\end{example}

The sharpness of the term $e^{-\lambda t}$ in \eqref{eq:Davies} follows from the long time heat kernel behavior, i.e. the exponential decay related to the bottom of the $\ell^2$ spectrum of Laplacian, which was first proved by Li \cite{Li86} on Riemannian manifolds and extended to graphs with unbounded Laplacians by Keller et al \cite[Corollary~5.6]{KellerLenzVogtWojciechowski11}.
\begin{thm} Let $(V,\mu,m)$ be a weighted graph. Then the minimal heat kernel satisfies
$$\lim_{t\to\infty}\frac{\log p_t(x,y)}{t}=-\lambda,$$
where $\lambda$ is the bottom of the $\ell^2$ spectrum of Laplacian. \end{thm}

\bibliography{Finitep}
\bibliographystyle{alpha}

\end{document}